\documentclass[11pt]{article}
\usepackage{amsthm, amsmath, amssymb, amsfonts, url, booktabs, tikz, setspace, fancyhdr, bm}
\usepackage{geometry}
\usepackage{hyperref, enumerate}
\usepackage[shortlabels]{enumitem}
\usepackage[babel]{microtype}
\usepackage[english]{babel}
\usepackage[capitalise]{cleveref}
\usepackage{comment}
\usepackage{bbm}
\usepackage{csquotes}
\usepackage{mathabx}
\usepackage{tikz}
\usepackage{graphicx}
\usepackage{float}
\usepackage{xcolor}
\usepackage{mathtools}
\usepackage{mathrsfs}

\usetikzlibrary{positioning, arrows.meta, shapes.geometric}

\counterwithin{figure}{section}

\newtheorem{theorem}{Theorem}[section]

\newtheorem{lemma}[theorem]{Lemma}
\newtheorem{cor}[theorem]{Corollary}

\newtheorem{claim}[theorem]{Claim}

\usetikzlibrary{decorations.pathmorphing}

\theoremstyle{definition}

\newtheorem*{defn-non}{Definition}

\newlist{Case}{enumerate}{2}
\setlist[Case, 1]{%
    label           =   {\bfseries Case \arabic*.},
    labelindent=1em ,labelwidth=1.3cm, labelsep*=1em, leftmargin =!
}
\setlist[Case, 2]{%
    label           =   {\bfseries Subcase \arabic{Casei}.\arabic*.},
    labelindent=-1em ,labelwidth=1.3cm, labelsep*=1em, leftmargin =!
}

\newenvironment{poc}{\begin{proof}[Proof of the claim]}{\end{proof}}

\usepackage{todonotes}

\title{A solution to a strengthened conjecture of Bukh, van Hintum and Keevash on additive bases} 
\author{
Zixiang Xu\thanks{Email: zxxu8023@qq.com.}
}

\date{}
\begin{document}
\maketitle
\begin{abstract}
Motivated by the change-of-domain problem for additive bases, Bukh, van Hintum and Keevash conjectured that if \(A,B\subseteq \mathbb{Q}^{n}\) and
\(
    \{\boldsymbol{e}_i+\boldsymbol{e}_j:1\le i\le j\le n\}\subseteq A+B,
\)
then \(|A|+|B|\ge 2n\). They further proposed the strengthened conjecture: if
\(
    |A|=n-t\), then \(|B|\ge n+\binom{t+1}{2}.
\)
Bukh also explicitly asked whether the same bounds hold for \(A,B\subseteq \mathbb{R}^{n}\) and an arbitrary basis \(S\) of \(\mathbb{R}^{n}\), under the assumption \(S+S\subseteq A+B\).

We prove the full strengthened statement over \(\mathbb{R}^{n}\): if \(S+S\subseteq A+B\) and \(|A|\le n-t\) with \(0\le t\le n-1\), then
\(
    |B|\ge n+\binom{t+1}{2},
\)
which is sharp for every basis \(S\) and every \(0\le t\le n-1.\)
 The proof is short, using edge contractions in a graph-theoretical framework and a new coloring lemma over \(\mathbb F_2^n\).

\end{abstract}

\section{Introduction}
A basic topic in additive number theory and additive combinatorics is to understand how efficiently a set can be generated by sums. Classical instances include Waring's problem~\cite{1964Sinica,1909Hilbert} and Goldbach-type problems~\cite{1973Sinica,2015Helfgott,1937Vinogradov}.  In finite settings, initiated by Erd\H{o}s and Newman~\cite{1977ErdosNewman}, one asks how small a set \(B\) can be while still satisfying \(A\subseteq B+B\), or more generally \(A\subseteq kB\).  This viewpoint has led to connections with extremal and geometric problems in additive combinatorics, including discrete Kakeya-type phenomena~\cite{2009IsrABS}; see also~\cite{1996Nathanson,2006Book}.

Recently, Bukh, van Hintum and Keevash~\cite{2025ATCAarith}, motivated by questions of Ruzsa~\cite{2024Ruzsa}, studied a natural change-of-domain problem: how does the minimum size of an additive basis change when the basis is required to lie in a more restrictive domain?  In the rational-versus-integer setting, they proved that every rational basis of size \(n\) can be replaced by an integer basis of size at most \(2n\), and constructed examples showing that this factor \(2\) is tight up to a lower-order error.

A key ingredient in their argument is a linear-algebraic vector model. Let \(\boldsymbol{e}_{1},\ldots,\boldsymbol{e}_{n}\) be the standard basis of \(\mathbb{R}^{n}.\) In the case of \(2\)-bases, this model asks for the minimum possible value of \(|A|+|B|\), where \(A,B\subseteq \mathbb{Q}^{n}\) and
\[
    \{\boldsymbol{e}_i+\boldsymbol{e}_j:1\le i\le j\le n\}\subseteq A+B.
\]
Bukh, van Hintum and Keevash~\cite{2025ATCAarith} proved the lower bound \(2n-O(\sqrt n)\) for \(|A|+|B|\) and conjectured that the sharp answer is exactly \(2n\)~\cite[Conjecture~13]{2025ATCAarith}.  They further proposed the refined extension that if \(|A|=n-t\), then one should have
\(
    |B|\ge n+\binom{t+1}{2}.
\)
On his open problem list, Bukh~\cite[Problem~12]{BukhProblems} asked explicitly whether this entire conjectural picture persists over \(\mathbb{R}^{n}\): if \(S\) is a basis of \(\mathbb{R}^{n}\) and \(A,B\subseteq \mathbb{R}^{n}\) satisfy \(S+S\subseteq A+B\), do the same bounds still hold?

Our main result answers this question affirmatively, including the refined extension.

\begin{theorem}\label{thm:main}
Let \(S\) be a basis of \(\mathbb{R}^{n}\), and let \(A,B\subseteq \mathbb{R}^{n}\). Suppose that
\(
    S+S\subseteq A+B.
\)
Let \(t\) be an integer with \(0\le t\le n-1\). If
\(
    |A|\le n-t,
\)
then
\(
    |B|\ge n+\binom{t+1}{2}.
\)
\end{theorem}

The bound is sharp for every basis \(S=\{\boldsymbol{s}_1,\dots,\boldsymbol{s}_n\}\) and \(0\le t\le n-1\).  Indeed, take
\(
    A=\{\boldsymbol{s}_1,\dots,\boldsymbol{s}_{n-t}\}
\)
and
\[
    B=\{\boldsymbol{s}_1,\dots,\boldsymbol{s}_n\}
    \cup
    \{\boldsymbol{s}_i+\boldsymbol{s}_j-\boldsymbol{s}_1:n-t<i\le j\le n\}.
\]
Then \(|A|=n-t\) and \(|B|=n+\binom{t+1}{2}\), with the displayed elements of \(B\) being distinct by the linear independence of \(S\). Moreover, \(S+S\subseteq A+B\): sums involving one of \(\boldsymbol{s}_1,\dots,\boldsymbol{s}_{n-t}\) are immediate, while the remaining sums are represented as
\(
    \boldsymbol{s}_i+\boldsymbol{s}_j
    =
    \boldsymbol{s}_1+(\boldsymbol{s}_i+\boldsymbol{s}_j-\boldsymbol{s}_1)\) for
    \(n-t<i\le j\le n\).

Since the standard basis of \(\mathbb{Q}^{n}\) is also a basis of \(\mathbb{R}^{n}\), Theorem~\ref{thm:main} immediately yields the vector-model conjecture of Bukh, van Hintum and Keevash~\cite[Conjecture~13]{2025ATCAarith} and its extension.

\begin{cor}\label{cor:BHK}
Let \(A,B\subseteq \mathbb{Q}^{n}\), and suppose that
\(
    \{\boldsymbol{e}_i+\boldsymbol{e}_j:1\le i\le j\le n\}\subseteq A+B.
\)
If \(0\le t\le n-1\) and \(|A|\le n-t\), then
\(
    |B|\ge n+\binom{t+1}{2}.
\)
\end{cor}
Combined with their reduction from the change-of-domain problem to the vector model, Corollary~\ref{cor:BHK} shows the factor \(2\) in~\cite[Theorem~1]{2025ATCAarith} is best possible already for \(2\)-bases: for every \(n\), there exists a set of integers with a rational \(2\)-basis of size at most \(n\) whose smallest integer \(2\)-basis has size exactly \(2n\).

\section{Proof of Theorem~\ref{thm:main}}
If \(B\) is infinite, then the conclusion is immediate. Thus we may assume that \(B\) is finite. It suffices to prove the lower bound in the case \(S=\{\boldsymbol{e}_1,\dots,\boldsymbol{e}_n\}\), where \(\boldsymbol{e}_1,\dots,\boldsymbol{e}_n\) are the standard basis vectors. Indeed, if \(S=\{\boldsymbol{s}_1,\dots,\boldsymbol{s}_n\}\) is an arbitrary basis, then there is an invertible linear map sending \(\boldsymbol{s}_i\) to \(\boldsymbol{e}_i\) for every \(i\). Applying this map preserves cardinalities and sumset inclusions.

If \(n=1\), then \(t=0\), and the inclusion \(S+S\subseteq A+B\) implies that \(B\) is nonempty. Hence \(|B|\ge 1=n+\binom{t+1}{2}\). We may therefore assume that \(n\ge 2\).

For every pair \(1\le i\le j\le n\), the containment \(S+S\subseteq A+B\) guarantees at least one representation of \(\boldsymbol{e}_i+\boldsymbol{e}_j\) as a sum of an element of \(A\) and an element of \(B\). Choose and fix one such representation:
\[
\boldsymbol{e}_i+\boldsymbol{e}_j=\boldsymbol{a}_{ij}+\boldsymbol{b}_{ij},\qquad
    \boldsymbol{a}_{ij}\in A,\quad \boldsymbol{b}_{ij}\in B.
\]
Let \(A_0\) be the set of distinct vectors among the \(\boldsymbol{a}_{ij}\), and let \(B_0\) be the set of distinct vectors among the \(\boldsymbol{b}_{ij}\). Put
\(
    p=|A_0|\) and \(q=|B_0|.\)
Then
\(
    p\le |A|\le n-t\) and \(q\le |B|.
\)

We construct a bipartite graph \(G\) as follows. Its left vertex set is a copy of \(A_0\), and its right vertex set is a copy of \(B_0\); these two vertex sets are regarded as disjoint copies, even if the same vector occurs on both sides. For each \(1\le i\le j\le n\), put an edge \(\varepsilon_{ij}\) from the left vertex \(\boldsymbol{a}_{ij}\) to the right vertex \(\boldsymbol{b}_{ij}\). Later contractions may create parallel edges, so we allow multigraphs from this point on.

Define \(\varphi:V(G)\to \mathbb{R}^{n}\) by setting \(\varphi(\boldsymbol{a})=\boldsymbol{a}\) for left vertices \(\boldsymbol{a}\), and \(\varphi(\boldsymbol{b})=-\boldsymbol{b}\) for right vertices \(\boldsymbol{b}\). Then for every \(1\le i\le j\le n\),
\begin{equation}\label{eq:fromPhiStrong}
    \varphi(\boldsymbol{a}_{ij})-\varphi(\boldsymbol{b}_{ij})
    =\boldsymbol{a}_{ij}+\boldsymbol{b}_{ij}
    =\boldsymbol{e}_i+\boldsymbol{e}_j.
\end{equation}

\begin{claim}\label{claim:forest}
    The subgraph of \(G\) whose edge set is
    \(\{\varepsilon_{11},\dots,\varepsilon_{nn}\}\)
    is a forest.
\end{claim}
\begin{poc}
    Suppose not.  Then this subgraph contains a cycle \(\boldsymbol{v}_0,\boldsymbol{v}_1,\dots,\boldsymbol{v}_r=\boldsymbol{v}_0
    \)
    whose edges are distinct.  For each \(1\le k\le r\), let the edge
    \(\boldsymbol{v}_{k-1}\boldsymbol{v}_k\) be \(\varepsilon_{i_k i_k}\).
    Since the edges on the cycle are distinct and there is only one diagonal
    edge \(\varepsilon_{ii}\) for each \(i\), the indices
    \(i_1,\dots,i_r\) are distinct.

    By \eqref{eq:fromPhiStrong}, for each \(k\) we have
    \(
        \varphi(\boldsymbol{v}_{k-1})-\varphi(\boldsymbol{v}_k)
        =
        \sigma_k\cdot 2\boldsymbol{e}_{i_k}
    \)
    for some \(\sigma_k\in\{\pm1\}\), depending on which direction we traverse
    the edge \(\varepsilon_{i_k i_k}\).  Summing these identities over
    \(k=1,\dots,r\), the left-hand side becomes zero. Hence
    \[
        \sum_{k=1}^r \sigma_k\cdot 2\boldsymbol{e}_{i_k}=0.
    \]
    This is a nontrivial linear relation among the distinct basis vectors
    \(\boldsymbol{e}_{i_1},\dots,\boldsymbol{e}_{i_r}\), contradicting the
    linear independence of
    \(\boldsymbol{e}_1,\dots,\boldsymbol{e}_n\).
\end{poc}

Now contract all diagonal edges \(\varepsilon_{ii}\), meaning that each connected component of the subgraph formed by these edges is identified to a single vertex. By Claim~\ref{claim:forest}, these \(n\) edges form a forest. Hence contracting them reduces the number of vertices by exactly the number of contracted edges. Thus the resulting graph \(H\) satisfies
\begin{equation}\label{eq:fromForestStrong}
    |V(H)|=p+q-n.
\end{equation}

We color the vertices of \(H\) by elements of the quotient group \(\mathbb{R}^{n}/2\mathbb{Z}^{n}\). Let
\[
    \pi:\mathbb{R}^{n}\to \mathbb{R}^{n}/2\mathbb{Z}^{n}
\]
be the quotient map. If \(X\) is a vertex of \(H\), choose any vertex \(\boldsymbol{x}\) of \(G\) that was contracted to \(X\), and assign to \(X\) the color \(\pi(\varphi(\boldsymbol{x}))\). This is well-defined. Indeed, if two vertices \(\boldsymbol{x},\boldsymbol{y}\in V(G)\) are contracted to the same vertex of \(H\), then they are connected in \(G\) by a path consisting only of diagonal edges. Along each diagonal edge \(\varepsilon_{ii}\), the two \(\varphi\)-values differ by \(\pm 2\boldsymbol{e}_i\). Hence \(\varphi(\boldsymbol{x})-\varphi(\boldsymbol{y})\in 2\mathbb{Z}^{n}\), and so
\( \pi(\varphi(\boldsymbol{x}))=\pi(\varphi(\boldsymbol{y})).
\)

Let \(\mathcal C\) be the set of colors appearing on the vertices of \(H\), and let \(\mathcal X\subseteq \mathcal C\) be the set of colors of those vertices of \(H\) which contain at least one left vertex of \(G\). Since every color in \(\mathcal X\) is witnessed by at least one left vertex of \(G\), and distinct colors require distinct witnesses, we have
\begin{equation}\label{eq:X-size}
    |\mathcal X|\le p\le n-t.
\end{equation}

For every off-diagonal edge \(\varepsilon_{ij}\) with \(i<j\), let \(X\) and \(Y\) be the vertices of \(H\) containing the left endpoint \(\boldsymbol{a}_{ij}\) and the right endpoint \(\boldsymbol{b}_{ij}\), respectively. By \eqref{eq:fromPhiStrong}, the color of \(X\) minus the color of \(Y\) is
\[
    \pi(\varphi(\boldsymbol{a}_{ij}))-\pi(\varphi(\boldsymbol{b}_{ij}))
    =\pi(\boldsymbol{e}_i+\boldsymbol{e}_j).
\]
Consequently,
\begin{equation}\label{eq:coverStrong}
    \{\pi(\boldsymbol{e}_i+\boldsymbol{e}_j):1\le i<j\le n\}\subseteq \mathcal X-\mathcal C.
\end{equation}

We shall take advantage of the following lemma.

\begin{lemma}\label{lem:coset}
Let \(n\ge 2\) and \(\Gamma\) be an abelian group, and let \(K\le \Gamma\) be a subgroup identified with \(\mathbb{F}_2^n\). Write
\(
    D_n=\{\boldsymbol{e}_i+\boldsymbol{e}_j:1\le i<j\le n\}\subseteq K.
\)
Let \(u\) be an integer with \(0\le u\le n\). Suppose that \(X\subseteq C\subseteq \Gamma\), \(|X|\le n-u\), and
\(
    D_n\subseteq X-C.
\)
Then
\(
    |C|\ge n+\binom{u}{2}.
\)
\end{lemma}

\begin{proof}[Proof of Lemma~\ref{lem:coset}]
First suppose that \(X,C\subseteq K=\mathbb{F}_2^n\). Since \(n\ge 2\), the set \(D_n\) is nonempty, and hence the containment \(D_n\subseteq X-C\) implies that \(X\) is nonempty. Choose \(\boldsymbol{x}_0\in X\) and replace \(X\) and \(C\) by \(X-\boldsymbol{x}_0\) and \(C-\boldsymbol{x}_0\), respectively. This preserves both the containment \(X\subseteq C\) and the difference set \(X-C\), so we may assume that \(\boldsymbol{0}\in X\).

Put \(m=|X|\), and let
\(
    W=\operatorname{span}_{\mathbb{F}_2}(X).
\)
Since \(\boldsymbol{0}\in X\), the space \(W\) is generated by the at most \(m-1\) nonzero elements of \(X\), and hence
\(
    \dim W\le m-1.
\)
Let
\(
    \rho:\mathbb{F}_2^n\to \mathbb{F}_2^n/W
\)
be the quotient map, and put \(s=\dim(\mathbb{F}_2^n/W)\). Then
\(
    s=n-\dim W\ge n-m+1.
\)
The point of passing to the quotient is that all elements of \(X\) collapse to the zero coset.  We shall use the following consequence of the covering assumption \(D_n\subseteq X-C\).
\begin{claim}\label{claim:containmentRho}
    \(
        \rho(D_n)\subseteq \rho(C).
    \)
\end{claim}
\begin{poc}
    Let \(\boldsymbol{d}\in D_n\).  By the assumption \(D_n\subseteq X-C\), we may write
    \(
        \boldsymbol{d}=\boldsymbol{x}-\boldsymbol{c}
    \)
    for some \(\boldsymbol{x}\in X\) and \(\boldsymbol{c}\in C\).  Since \(X\subseteq W\), we have \(\rho(\boldsymbol{x})=\boldsymbol{0}\), and therefore
    \[
        \rho(\boldsymbol{d})
        =
        \rho(\boldsymbol{x})-\rho(\boldsymbol{c})
        =
        -\rho(\boldsymbol{c})
        =
        \rho(\boldsymbol{c}),
    \]
    where the last equality uses that \(\mathbb{F}_2^n/W\) has exponent \(2\).  Thus \(\rho(\boldsymbol{d})\in \rho(C)\), as claimed.
\end{poc}

We now count the elements of \(C\) according to the \(W\)-cosets they occupy.  Since \(X\subseteq C\cap W\), the zero coset of \(W\) already contains at least the \(m\) elements of \(X\).  Moreover, by Claim~\ref{claim:containmentRho}, every nonzero element of \(\rho(D_n)\) corresponds to a nonzero \(W\)-coset that meets \(C\).  Hence
\[
    |C|\ge m+|\rho(D_n)\setminus\{\boldsymbol{0}\}|.
\]

Since the vectors
\(
    \rho(\boldsymbol{e}_1),\dots,\rho(\boldsymbol{e}_n)
\)
span \(\mathbb{F}_2^n/W\), we may choose indices
\(
    1\le i_1<\cdots<i_s\le n
\)
such that
\(  \rho(\boldsymbol{e}_{i_1}),\dots,\rho(\boldsymbol{e}_{i_s})
\)
form a basis of \(\mathbb{F}_2^n/W\).  The pairwise sums of these \(s\) basis vectors are nonzero and pairwise distinct. For every \(1\le a<b\le s\),
\[
    \rho(\boldsymbol{e}_{i_a})+\rho(\boldsymbol{e}_{i_b})
    =
    \rho(\boldsymbol{e}_{i_a}+\boldsymbol{e}_{i_b})
    \in \rho(D_n).
\]
Hence
\(
    |\rho(D_n)\setminus\{\boldsymbol{0}\}|
    \ge \binom{s}{2}.
\)
Moreover, since \(s=n-\dim W\ge n-m+1\) and \(m\le n-u\), we have
\[
    |C|\ge m+\binom{s}{2}\ge m+\binom{n-m+1}{2}= n+\binom{n-m}{2}\ge n+\binom{u}{2}.
\]
It remains to reduce the general case \(X\subseteq C\subseteq \Gamma\) to the case already proved, in which both sets lie inside \(K=\mathbb{F}_2^n\). The key observation is that, since every element of \(D_n\) lies in \(K\), any representation of an element of \(D_n\) as a difference \(\boldsymbol{x}-\boldsymbol{c}\), with \(\boldsymbol{x}\in X\) and \(\boldsymbol{c}\in C\), must use two elements from the same coset of \(K\).  We therefore fold each \(K\)-coset meeting \(X\) back onto \(K\).

For each coset \(\alpha+K\) meeting \(X\), choose an element
\(
    \boldsymbol{x}_\alpha\in X\cap(\alpha+K),
\)
and define
\[
    X_\alpha'
    =
    (X\cap(\alpha+K))-\boldsymbol{x}_\alpha
    \subseteq K,
    \qquad
    C_\alpha'
    =
    (C\cap(\alpha+K))-\boldsymbol{x}_\alpha
    \subseteq K.
\]
Let
\(
    X'=\bigcup_\alpha X_\alpha'\) and
    \(
    C'=\bigcup_\alpha C_\alpha',
\)
where the unions range over all cosets of \(K\) meeting \(X\).  Since \(X\subseteq C\), we have
\(
    X'\subseteq C'\subseteq K.
\)
Moreover,
\(
    |X'|\le |X|\) and
    \(
    |C'|\le |C|,
\)
because translation preserves cardinalities within each coset, while sets coming from distinct cosets may overlap after being translated into \(K\). We then claim that the covering property of \(D_n\) is preserved under this folding operation.
\begin{claim}\label{claim:DXC}
    \(D_n\subseteq X'-C'.\)
\end{claim}
\begin{poc}
    Let \(\boldsymbol{d}\in D_n\).  Since \(D_n\subseteq X-C\), we write
    \(
        \boldsymbol{d}=\boldsymbol{x}-\boldsymbol{c}
    \)
    for some \(\boldsymbol{x}\in X\) and \(\boldsymbol{c}\in C\).  As \(\boldsymbol{d}\in K\), the elements \(\boldsymbol{x}\) and \(\boldsymbol{c}\) lie in the same coset \(\alpha+K\).  Therefore
    \(
        \boldsymbol{x}-\boldsymbol{x}_\alpha\in X_\alpha'\) and
        \(
        \boldsymbol{c}-\boldsymbol{x}_\alpha\in C_\alpha',
    \)
    which further implies that
    \(
        \boldsymbol{d}
        =
        (\boldsymbol{x}-\boldsymbol{x}_\alpha)
        -
        (\boldsymbol{c}-\boldsymbol{x}_\alpha)
        \in X'-C'.
    \) This finishes the proof.
\end{poc}
By construction and Claim~\ref{claim:DXC}, the pair \(X',C'\) lies inside \(K\) and satisfies
\(
    X'\subseteq C',\) \(
    |X'|\le |X|\le n-u\) and
    \(D_n\subseteq X'-C'.
\)
Therefore the special case already proved for subsets of \(K=\mathbb{F}_2^n\) applies to \(X'\) and \(C'\), and yields
\(
    |C'|\ge n+\binom{u}{2}.
\)
Finally, the folding operation used to construct \(C'\) can only decrease cardinality, so 
\(
    |C|\ge |C'|\ge n+\binom{u}{2}.
\)
This completes the proof.
\end{proof}

We now apply Lemma~\ref{lem:coset} with
\( \Gamma=\mathbb{R}^{n}/2\mathbb{Z}^{n}\)
and \(K=\mathbb{Z}^{n}/2\mathbb{Z}^{n}\cong \mathbb{F}_2^n.\)
Moreover, set
\(
    X=\mathcal{X},\) \(C=\mathcal{C},\) and \(u=t.\)
Under the natural identification of \(K\) with \(\mathbb{F}_2^n\), the set \(D_n\) becomes
\[
    \{\pi(\boldsymbol{e}_i+\boldsymbol{e}_j):1\le i<j\le n\}.
\]
By definition, \(\mathcal X\subseteq \mathcal C\).  Moreover, \eqref{eq:X-size} gives
\(
    |\mathcal X|\le n-t,
\)
while \eqref{eq:coverStrong} gives
\(
    D_n\subseteq \mathcal X-\mathcal C.
\)
Hence all hypotheses of Lemma~\ref{lem:coset} are satisfied, and therefore
\(
    |\mathcal C|\ge n+\binom{t}{2}.
\)

Since \(\mathcal C\) is the set of colors appearing on the vertices of \(H\), we have
\(
    |\mathcal C|\le |V(H)|.
\)
Combining this with \eqref{eq:fromForestStrong}, we obtain
\[
    p+q-n
    =
    |V(H)|
    \ge
    |\mathcal C|
    \ge
    n+\binom{t}{2}.
\]
Thus
\(
    q\ge 2n+\binom{t}{2}-p.
\)
Finally, using \(p\le n-t\) and \(q\le |B|\), we get
\[
    |B|
    \ge q
    \ge
    2n+\binom{t}{2}-(n-t)
    =
    n+t+\binom{t}{2}
    =
    n+\binom{t+1}{2}.
\]
This completes the proof.

\section*{Acknowledgement}
The author is grateful to Yang Shu for many helpful comments on the proof.

\bibliographystyle{abbrv}
\bibliography{BHK}
\end{document}